\documentclass[11pt,a4paper]{article}
\usepackage[T1]{fontenc}
\usepackage[utf8]{inputenc}
\usepackage{lmodern,microtype}
\usepackage[margin=25mm]{geometry}
\usepackage{amsmath,amssymb,amsthm,mathtools}
\usepackage{enumitem}
\usepackage{graphicx}
\usepackage{booktabs,tabularx}
\usepackage[colorlinks=true,linkcolor=blue,citecolor=blue,urlcolor=blue]{hyperref}
\newtheorem{theorem}{Theorem}[section]
\newtheorem{proposition}[theorem]{Proposition}
\newtheorem{lemma}[theorem]{Lemma}
\setlist{itemsep=2pt,topsep=4pt}
\allowdisplaybreaks[2]
\newcommand{\norm}[1]{\|#1\|}

\newcommand{\bbR}{\mathbb{R}}  \newcommand{\bbZ}{\mathbb{Z}}  
\newcommand{\Pp}{\mathbb{P}}
\newcommand{\surcharge}{\mathfrak{s}}

\title{Non-perturbative finite connectivities in supercritical Bernoulli percolation on \(\mathbb{Z}^d\) with \(d\geq 3\)}

\author{Yacine Aoun \and Kamil Khettabi\thanks{
Section de mathématiques, Université de Genève, Switzerland.
Email: \texttt{kamil.khettabi@unige.ch}
}
}

\usepackage{xcolor}
\begin{document}

\maketitle
\begin{abstract}
    We solve a longstanding problem in Bernoulli percolation: we derive sharp Ornstein--Zernike asymptotics for the probability that two distant points are connected by a finite cluster in the whole supercritical regime. Our proof has two main ingredients. First, the renewal structure of~\cite{FH} yields an \emph{a priori} bound on cluster volume. Second, we adapt the coarse-graining and irreducible-cluster framework of~\cite{CIV08} to finite supercritical connections. The main distinction from the subcritical analysis is that our argument \emph{does not require} a separation of masses between irreducible and unrestricted finite connections. The volume bound and a modified coarse-graining construction instead provide stretched-exponential tails for the effective random-walk increments, which suffice to obtain the sharp asymptotics.
\end{abstract}
\section{Introduction}
\subsection{Context}
The problem of establishing a sharp control on general two-point functions (covariance, connectivity functions, Green functions) has a long history in mathematics and physics. It can be traced back to the original work of Ornstein and Zernike on density fluctuations in fluids~\cite{OZ14}. The resulting Ornstein--Zernike (OZ) asymptotics describe exponential decay with a polynomial prefactor of order $\|x\|^{-(d-1)/2}$, where $\norm{\cdot}$ is the Euclidean norm. Since then, the OZ asymptotics became a cornerstone of the study of two-point functions in statistical mechanics: see for instance~\cite{F64,CF71,AK77,PL78}.

For Bernoulli percolation, the first non-perturbative OZ asymptotics for the probability of two distant points to be in the same cluster were derived for the subcritical regime by Campanino, Chayes and Chayes~\cite{CCC91} for axial directions, and later extended to arbitrary directions by Campanino and Ioffe~\cite{CI02}. The theory was then extended for the finite-range Ising model above the critical temperature~\cite{CIV03} and for the subcritical finite-range FK percolation model by Campanino, Ioffe and Velenik~\cite{CIV08}. For the infinite-range Ising model with exponentially decaying interactions, the OZ asymptotics can fail~\cite{AOV23, AIOV21}. The OZ asymptotics were established under appropriate assumptions in the subcritical and non-saturated regime by Aoun, Ott and Velenik~\cite{AOV24}. At the saturated point (where the transition between the saturated and unsaturated regime occurs), we established~\cite{AK23} the OZ asymptotics up to constants on \(\bbZ\) under appropriate assumptions.

For the supercritical counterpart, Bricmont and Fröhlich~\cite{BF85} established the OZ asymptotics for spin covariances in the Ising model at very low temperatures on \(\bbZ^d\) with \(d\geq 3\). For Bernoulli percolation, the two-point function is the probability of two points of being connected by a finite cluster. On \(\bbZ^2\), Campanino, Ioffe and Louidor~\cite{CIL10} derived a different type of asymptotics (of order \( \|x\|^{-2}\) instead of \(\|x\|^{-1/2}\) that would be given by the OZ formula) in the whole supercritical phase. In dimension \(d\geq 3\), the OZ asymptotics were derived for \(p\) sufficiently close to \(1\) by Braga, Procacci and Sanchis~\cite{BPS04} for the axial directions and by Campanino and Gianfelice~\cite{CG11} for an arbitrary direction (and later extended to FK percolation~\cite{CG15}). More recently, Fribergh and Hammond~\cite{FH} proved sharp point-to-plane asymptotics in the whole supercritical phase using a renewal structure with stretched-exponential tails on \(\bbZ^d\) with \(d\geq 3\). In the present work, we prove the last remaining asymptotics for nearest neighbor Bernoulli percolation on \(\bbZ^d\): we establish the OZ asymptotics in the whole supercritical regime in dimension \(d\geq 3\).

\subsection{Basic definitions and properties}
Consider nearest-neighbour Bernoulli bond percolation on $\mathbb Z^d$ with \(d\geq 3\): each edge of \(\bbZ^d\) is declared \emph{open} with probability \(p\in (0,1)\) and \emph{closed} with probability \(1-p\), independently of all other edges. We denote the corresponding measure by \(\mathbb P_p\). We say that \(x,y\in\bbZ^d\) are connected if there exists a sequence of open edges from \(x\) to \(y\). We denote this event by \(\lbrace x\leftrightarrow y\rbrace\). There exists a critical parameter \(p_c(d)\) such that almost surely there is not an open infinite connected component when \(p<p_c\) and almost surely there is an open infinite connected component when \(p>p_c\), see~\cite{G99}. From now on, we consider \(p\in(p_c,1)\). Write $C_0$ for the open cluster of $0$,
regarded as a graph, and $|C_0|$ for its number of vertices. Set
\begin{align*}
 G_p(x)&=\mathbb P_p(0\leftrightarrow x,\ |C_0|<\infty),\\
 \xi_p(z)&=-\lim_{k\to\infty}\frac1k\log G_p(kz)
 \quad(z\in\mathbb Z^d).
\end{align*}
The limit above exists and its continuous homogeneous extension defines a norm on \(\bbR^d\), see~\cite{CCGKS89,G99, CG15}. We will denote the Euclidean norm by \(\|\cdot\|\).

\subsection{Main results}

Our first theorem establishes a control of the volume of the cluster of two distant points at the exponential scale: 
\begin{theorem}\label{thm:main}
There exist $K,C,c\in(0,\infty)$, depending only on $p,d$, such that,
for every $x\in\mathbb Z^d\setminus\{0\}$,
\begin{equation}\label{eq:main}
 \mathbb P_p(0\leftrightarrow x,\ K\norm{x}<|C_0|<\infty)
 \le C\exp\left\{-\xi_p(x)-c\|x\|^{\frac{d-1}{2d}}\right\}.
\end{equation}
\end{theorem}
The proof of this theorem will use the setting of~\cite{FH}. We prove this theorem in Section~\ref{section:Hammond}. Our second theorem is the main result of this paper:
\begin{theorem}\label{theorem:OZ}
There exists a continuous function $\Psi_p:\mathbb S^{d-1}\to(0,\infty)$ such that, uniformly in \(x\in\bbZ^d\) as \(\|x\|\rightarrow\infty\),
\[
 \mathbb P_p(0\leftrightarrow x, |C_0|<\infty)=
 \frac{\Psi_p(x/\|x\|)}{\|x\|^{(d-1)/2}}e^{-\xi_p(x)}(1+o_{\|x\|\rightarrow\infty}(1)),
\]
\end{theorem}
The proof of this theorem goes through an extension of the modern Ornstein-Zernike theory. We prove this theorem in Section~\ref{section:OZ}.
\subsection{Idea of the proof}

As already mentioned, our proof has two ingredients : a control on the volume of the cluster \(C_0\) on \(\{0\leftrightarrow x\}\) and a modification of the modern OZ theory. Let us explain quickly the main ideas :

For the first result, fix $x\ne0$ and choose $\ell\in\mathbb S^{d-1}$ such that
$\zeta_\ell\,\ell\cdot x=\xi_p(x)$, where $\zeta_\ell$
is the point-to-plane decay rate of \cite{FH}.
Set $u=\ell\cdot x$ and
$\mathcal H_u=\{z\in\mathbb Z^d:\ell\cdot z\ge u\}$.
Since $x\in\mathcal H_u$,
\[
 \mathbb P_p(0\leftrightarrow x,\ K\norm{x}<|C_0|<\infty)
 \le
 \mathbb P_p(0\leftrightarrow\mathcal H_u,\
                      K\norm{x}<|C_0|<\infty).
\]
The choice of $\ell$ ensures that this comparison preserves
the exponential cost of the connection.

Using the string decomposition of \cite{FH}, we then extract
a factor $e^{-\zeta_\ell u}=e^{-\xi_p(x)}$ from the RHS.
We then show that
\[
 \mu_\ell(V_1\ge v)\le Ce^{-cv^{\alpha/2}}.
\]
A routine large-deviation estimate then yields
\[
 \mathbb P_p(0\leftrightarrow\mathcal H_u,\
                      K\norm{x}<|C_0|<\infty)
 \le C e^{-\xi_p(x)-c\norm{x}^\beta},
\]
We may therefore restrict to clusters of volume at most
$K\norm{x}$ up to a stretched exponential error.

Then, we develop a OZ theory for the supercritical finite connections. The main skeleton of the proof stays the same : to every cluster, we associate a skeleton and we derive an appropriate energy estimate for every skeleton. The standard problem when deriving asymptotics in this way is that it might happen that some parts of the cluster are very densely connected, and hence on a local level it might be impossible to extract any energy. To circumvent this problem, we notice that the a priori control on the volume of the connected component \emph{prevents} this from happening too often, yielding overall a 'missing' exponential cost that can be made (on an exponential scale) arbitrarily small, hence removing the problem altogether. We establish that, up to an exponential error, there is a positive density of cone-points. In a sense, this result establishes a \emph{mass-gap} of some particular connections. Combined with the first result above, this allows us to construct a directed random walk with stretched exponential tails, which are enough to sharply derive the wanted asymptotics.

\subsection{Extensions}

Our results suggest several directions for further research. A natural extension would be to prove the OZ asymptotics in the supercritical phase for a general finite-range Bernoulli percolation model. This would require an extension of the renewal theory from~\cite{FH} as well as extending the OZ theory for this setting (the latter seems to be a more direct adaptation). 

Another natural direction would be extending our results for dependent percolation models such as FK percolation. This would extend the perturbative results of~\cite{CG15} in the whole supercritical regime. 

Another (potentially more challenging) extension would be to apply our framework to study the low-temperature covariances in Ising and Potts models. For the Ising model, possible approaches include the random current representation~\cite{A82} and the recent double cluster swapping representation introduced by van Engelenburg, Heeney and Lis~\cite{EHL26}, which provides a percolation representation of truncated correlations.

\subsection{Definitions and inputs}
For a finite graph $A$, write $\operatorname{Vert}(A)$ and $E(A)$ for its vertex and edge sets, and $v(A):=|\operatorname{Vert}(A)|$. Inclusions of graphs in regions and intersections with regions refer to their vertex sets unless edges are explicitly mentioned. For a vertex set $D$, $A\setminus D$ denotes the subgraph induced by $\operatorname{Vert}(A)\setminus D$. For an edge viewed as an unordered pair of endpoints, $|e\cap D|$ is the number of its endpoints in $D$. Unless specified otherwise, $c,C>0$ are constants whose values may change from line to line.
\\We will need the following supercritical estimates
\begin{align}
 \Pp(v\le |C_0|<\infty)&\le Ce^{-cv^\alpha},\notag\\
 \Pp\bigl(|C_0|<\infty,\ \max_{z\in C_0}\|z\|\ge u\bigr)&\le Ce^{-cu}.
 \label{eq:sharpness}
\end{align}
where $\alpha=(d-1)/d=:2\beta$. The first follows from Kesten--Zhang and Grimmett--Marstrand
\cite{KZ90,GM90}, for the second see \cite{CCGKS89,GM90}.

A direct consequence of \cite[Theorem~1.2]{BHK06} is that
\begin{equation}\label{eq:pointwise}
 G_p(x)\le e^{-\xi_p(x)}.
\end{equation}

We denote by \(U_p\) the unit ball associated to the norm \(\xi_p\) and \(K_p\) its polar body
\begin{align*}
\mathsf U_p &= \{y:\xi_p(y)\le1\}\\
 \mathsf K_p &=\{t\in\mathbb R^d:t\cdot y\le\xi_p(y)
                              \text{ for all }y\in\mathbb R^d\}.
\end{align*}
\section{A priori volume control for long finite clusters}\label{section:Hammond}
Firstly, we are going to introduce some definitions and recall few results from~\cite{FH}. The, we prove Theorem~\ref{thm:main}.
\subsection{Definitions and inputs from~\texorpdfstring{\cite{FH}}{FH}}
Fix now $\ell\in\mathbb S^{d-1}$ and choose a vector $e\in\{\pm e_1,\ldots,\pm e_d\}$
such that
\[
 r=\ell\cdot e=\max_i|\ell_i|\in[d^{-1/2},1].
\]
We are going to import some of the technology developed in \cite{FH}, which is a crucial input for the main result.
Let $A$ be a finite nonempty connected subgraph of $\mathbb Z^d$.
Define the $\ell$-height of a vertex $z$ to be equal to $\ell\cdot z$. Introduce
\begin{align*}
 b_\ell(A)&:=\min_{\mathrm{lex}}
   \operatorname*{arg\,min}_{z\in\operatorname{Vert}(A)}\ell\cdot z,\\
 t_\ell(A)&:=\min_{\mathrm{lex}}
   \operatorname*{arg\,max}_{z\in\operatorname{Vert}(A)}\ell\cdot z,
\end{align*}
where $\min_{\mathrm{lex}}$ selects the smallest, for the lexicographic order, vertex among those attaining the minimal / maximal height, so that $b_\ell(A)$ and $t_\ell(A)$ are unique. 
Define
\[
 H_\ell(A):=\ell\cdot\bigl(t_\ell(A)-b_\ell(A)\bigr)=\max_{z\in\operatorname{Vert}(A)}\ell\cdot z -\min_{z\in\operatorname{Vert}(A)}\ell\cdot z.
\]
For an edge $a$, let $a^\circ$ be the corresponding open line segment in \(\bbR^d\). Define the edge sets
\begin{align*}
 U_h&:=\{a:a^\circ\cap\{z:\ell\cdot z\ge h\}\ne\varnothing\},\\
 L_h&:=\{a:a^\circ\subset\{z:\ell\cdot z<h\}\},\\
 \mathcal S_{a,b}&:=U_a\cap L_b.
\end{align*}
A vertex $z$ is a renewal point of $A$ if
\[
 \{v\in\operatorname{Vert}(A):\ell\cdot z-r<\ell\cdot v\le\ell\cdot z\}=\{z\}.
\]
A string is a finite connected graph $S$, of strictly positive height,
whose edges lie in $\mathcal S_{\ell\cdot b_\ell(S),\ell\cdot t_\ell(S)}$
and whose base and head are renewal points. An
\emph{irreducible string}, called a sausage in \cite{FH}, has no other
renewal point. All strings below have base $0$.
Let $O_\ell(S)$ be the event that every edge of $S$ is open and every other edge of its slab with at least one endpoint in $\operatorname{Vert}(S)$ is closed. For an interval $I\subset\mathbb R$, let $\operatorname{Str}_I$ be the union of these events over strings based at $0$ whose height belongs to $I$. Define
$$w_\ell(S):=\mathbb P_p(O_\ell(S)).$$
It was shown in \cite[Proposition~1.7]{FH} that
\begin{equation}
 \zeta_\ell=-\lim_{u\to\infty}\frac1u
             \log\mathbb P_p(\operatorname{Str}_{[u,u+r)})\in(0,\infty).
\end{equation}
By \cite[Lemma~2.7 and Proposition~2.9]{FH}, for every $S$ irreducible one has that
\begin{equation}
 \mu_\ell(S) :=e^{\zeta_\ell H_\ell(S)}w_\ell(S)
\end{equation}
is a probability measure on the set of irreducible strings.
Every string is uniquely representable as a concatenation of irreducible strings, and
its weight is the product of the weights of the irreducible strings that compose it.

Let $S_1,S_2,\ldots$ be independent with law $\mu_\ell$, and define 
$H_i=H_\ell(S_i)$, $V_i=v(S_i)$,
$T_N=\sum_{i\le N}H_i$, $W_N=\sum_{i\le N}V_i$.
With these notations, for any interval $I$ of finite length $L$,
\begin{equation}
\sum_{S:\,H_\ell(S)\in I}
e^{\zeta_\ell H_\ell(S)}w_\ell(S)
=
\mathbb E_{\mu_\ell}
\left[\sum_{N\geq1}\mathbf{1}_{\{T_N\in I\}}\right]
\leq 1+\frac{L}{r},
\end{equation}
where the sum on the left runs over all strings with base $0$.
It is immediate to see that
\begin{equation}\label{eq:coutfermeture}
 w_\ell(S)\le (1-p)^{-4d}\mathbb P_p(C_0=S).
\end{equation}
Note that there exist constants $\zeta_-,\zeta_+\in(0,\infty)$, depending only on $p,d$, such that for every $\ell\in\mathbb{S}^{d-1}$ one has \(0<\zeta_-\leq\zeta_\ell\leq\zeta_+<\infty\).
\\Indeed, for the lower bound, let $c_0>0$ be the constant in the second
estimate of \eqref{eq:sharpness}. Then, 
\[
\begin{aligned}
\mathbb P_p\bigl(\operatorname{Str}_{[u,u+r)}\bigr)
&\leq
(1-p)^{-4d}
\mathbb P_p\left(
|C_0|<\infty,\,
\max_{z\in C_0}\|z\|\geq u
\right)
\\
&\leq C(1-p)^{-4d}e^{-c_0u}.
\end{aligned}
\]
Therefore, one can choose \(\zeta_\ell\geq c_0=:\zeta_-.\)
The upper bound is trivial.
It is proven in \cite[Lemma~2.18]{FH} that the tail of the head is stretched-exponential, in the sense that
\begin{equation}
 \mu_\ell(H_1\ge u)\le C e^{-c\sqrt u}.
\end{equation}
Here $c,C$ can be chosen independently of $\ell$, as specified before
\cite[Lemma~2.18]{FH}.
\subsection{Proof of Theorem~\ref{thm:main}}
In this subsection, we are going to prove Theorem~\ref{thm:main}. We are firstly going to prove few technical lemmas. The proof of Theorem~\ref{thm:main} can be found at the very end of this subsection. Let us explain the strategy of the proof. Given the irreducible decomposition explained above, we reason as follows : firstly, we prove that the independent random variables \(V_i\) have stretched exponential tails under the tilted measure, this is directly proved below. We then derive some basic estimates on the large deviation of the sum of these variables. Finally, we combine all of these to prove the main estimate.
Recall that \(\beta:=\alpha/2=\frac{d-1}{2d}\).
By \eqref{eq:coutfermeture} and the first bound in \eqref{eq:sharpness},
\(\sum_{S:\,v(S)\ge v}w_\ell(S)\le C e^{-c_0v^\alpha}.\) Choosing $a>0$ small enough, one has
\begin{align}
 \mu_\ell(V_1\ge v)
 &\le\mu_\ell(H_1\ge av^\alpha)
       +e^{\zeta_+av^\alpha}\sum_{S:\,v(S)\ge v}w_\ell(S)\notag\\
 &\le C e^{-c v^\beta}+C e^{-c_0v^\alpha/2}
 \le C e^{-c v^\beta}.
\end{align}
In particular, there exists $\kappa>0$ such that
\begin{align}
 m:=\sup_\ell\mathbb E_{\mu_\ell}V_1&<\infty,\notag\\
 \sup_\ell\mathbb E_{\mu_\ell}[V_1^2e^{\kappa V_1^\beta}]&<\infty.
 \label{eq:moments}
\end{align}

\begin{lemma}\label{lem:ld}
For each fixed $A<\infty$, there are $K,C,c>0$, independent of $\ell$,
such that for $n\ge1$ and $1\le N\le An$,
\begin{align*}
 \mathbb P_{\mu_\ell}(W_N>Kn)&\le C e^{-cn^\beta},\\
 \mathbb E_{\mu_\ell}[W_N\mathbf1_{\ W_N>Kn}]&\le Cn e^{-cn^\beta}.
\end{align*}
\end{lemma}
\begin{proof}
Take $U_i:=V_i\mathbf1_{\{V_i\le n\}}$ and
$\lambda:=\kappa n^{\beta-1}$.
By union bound and Markov inequality, for $K$ big enough one has
\begin{align*}
 \mathbb P_{\mu_\ell}(W_N>Kn)
 &\le N\mathbb P_{\mu_\ell}(V_1>n)
   +\exp\{-\lambda Kn+N\lambda m+CN\lambda^2\}\\
 &\le C e^{-c'n^\beta},
\end{align*}
Since $\mathbb E_{\mu_\ell} W_N^2\le C N^2$, the second bound follows immediately.
\end{proof}

For a finite connected graph $A$ with $b_\ell(A)=0$, let $z=t_\ell(A)$.
Add the two edges\(\{-e,0\}\) and \(\{z,z+e\} \) and translate the resulting graph by $e$. Denote it by $S(A)$. Clearly, this new graph is a string with $H_\ell(S(A))>H_\ell(A)$ and $v(S(A))> v(A)$.
Moreover, deleting the two added vertices recovers $A+e$, so this map is injective. Therefore
\begin{equation}
 w_\ell(S(A))\ge p^2(1-p)^{4d}\mathbb P_p(C_0=A).
\end{equation}
Now, remark that for any nonnegative translation-invariant function $F$ of finite graphs, it is easy to see that \(\mathbb E_p[F(C_0)\mathbf1_{\{|C_0|<\infty\}}]=\sum_{A:\,b_\ell(A)=0}v(A)\mathbb P_p(C_0=A)F(A).\) Therefore, for $u,v\ge0$,
\begin{align}
 \mathbb P_p\bigl(H_\ell(C_0)\ge u,\ v<|C_0|<\infty\bigr)\notag &\le C\sum_S v(S)w_\ell(S) \mathbf1_{\{H_\ell(S)\ge u,\ v(S)>v\}}\notag\\
 &\le C\sum_{N\ge1}\mathbb E_{\mu_\ell}\!\left[e^{-\zeta_\ell T_N}W_N\mathbf1_{\{T_N\ge u,\ W_N>v\}}\right],
 \label{eq:master}
\end{align}
\begin{lemma}
There exist $K,C,c>0$, depending only on $p,d$, such that,
for every $\ell\in\mathbb S^{d-1}$, $n\ge1$ and $0\le u\le n$,
\[
 \mathbb P_p\bigl(H_\ell(C_0)\ge u,\ Kn<|C_0|<\infty\bigr)
 \le C e^{-\zeta_\ell u-cn^\beta}.
\]
\end{lemma}
Take $n\ge1$, $0\le u\le n$, and let $A=2\sqrt d$.
Choose $K$ as in Lemma~\ref{lem:ld}. In \eqref{eq:master} with $v=Kn$,
first restrict to $T_N\le2n$. Since $H_i\ge r\ge d^{-1/2}$, this forces
$N\le An$, and therefore
\begin{align}
\sum_{N\ge1}\mathbb E_{\mu_\ell}\!\left[
 e^{-\zeta_\ell T_N}W_N \mathbf1_{u\le T_N\le2n,\ W_N>Kn}\right]\notag &\le e^{-\zeta_\ell u}\sum_{N\le An}\mathbb E_{\mu_\ell}[W_N\mathbf1_{\ W_N>Kn}]\notag\\
       &\le C e^{-\zeta_\ell u-cn^\beta}.
 \label{eq:moderate-height}
\end{align}

On $T_N>2n$, one has that $u\le n$ so
\[
 \zeta_\ell T_N\ge
 \zeta_\ell u+\frac{\zeta_\ell n}{2}+\frac{\zeta_\ell T_N}{4}.
\]

By \eqref{eq:moments} and the independence of the heights, we obtain
\[
 \sum_{N\ge1}\mathbb E_{\mu_\ell}
           [W_Ne^{-\zeta_\ell T_N/4}]
 =\frac{\mathbb E_{\mu_\ell}[V_1e^{-\zeta_\ell H_1/4}]}
        {(1-\mathbb E_{\mu_\ell}e^{-\zeta_\ell H_1/4})^2}
 \le C.
\]
Thus the contribution of $T_N>2n$ is at most
$C e^{-\zeta_\ell u-\zeta_-n/2}$. Together with
\eqref{eq:master} and \eqref{eq:moderate-height}, this proves that, for $0\le u\le n$ , one has
\begin{equation}
 \mathbb P_p(H_\ell(C_0)\ge u,\ Kn<|C_0|<\infty)
 \le C e^{-\zeta_\ell u-cn^\beta},\label{eq:span-volume}
\end{equation}
Finally, we are going to need the following technical result for the proof of Theorems~\ref{thm:main} and~\ref{theorem:OZ}. Remember
\[
 \mathsf K_p=\{t\in\mathbb R^d:t\cdot y\le\xi_p(y)
                              \text{ for all }y\in\mathbb R^d\}.
\]
\begin{lemma}\label{lem:point-plane-duality}
For every $\ell\in\mathbb S^{d-1}$, $\zeta_\ell\ell\in\partial\mathsf K_p$.
Consequently, every $t\in\partial\mathsf K_p$ satisfies
$t=\zeta_\ell\ell$ with $\ell=t/\|t\|$, and
\[
 \sum_{y\in\mathbb Z^d}e^{t\cdot y}G_p(y)=\infty.
\]
\end{lemma}
\begin{proof}
Fix $\ell$ and $0<\varepsilon<\zeta_\ell$. By \eqref{eq:master} with $v=0$,
\[
 \mathbb P_p(H_\ell(C_0)\ge u,\ |C_0|<\infty)
 \le Ce^{-(\zeta_\ell-\varepsilon)u}
      \sum_{N\ge1}\mathbb E_{\mu_\ell}[W_Ne^{-\varepsilon T_N}].
\]
Independence and \eqref{eq:moments} give
\[
 \sum_{N\ge1}\mathbb E_{\mu_\ell}[W_Ne^{-\varepsilon T_N}]
 =\frac{\mathbb E_{\mu_\ell}[V_1e^{-\varepsilon H_1}]}
        {(1-\mathbb E_{\mu_\ell}[e^{-\varepsilon H_1}])^2}<\infty,
\]
since $H_1\ge r>0$. Applying this bound to $G_p(kz)$ for
$\ell\cdot z>0$, then letting $k\to\infty$ and $\varepsilon\downarrow0$,
yields $\xi_p(z)\ge\zeta_\ell\ell\cdot z$.
For $\ell\cdot z\le0$ this is immediate. Homogeneity and continuity therefore give $$\zeta_\ell\ell\in\mathsf K_p.$$
On the other hand, \eqref{eq:coutfermeture} and the renewal decomposition imply
\[
 \sum_y e^{\zeta_\ell\ell\cdot y}G_p(y)
 \ge(1-p)^{4d}\sum_{S\text{ string}}
       e^{\zeta_\ell H_\ell(S)}w_\ell(S)
 =(1-p)^{4d}\sum_{N\ge1}1=\infty.
\]
By \eqref{eq:pointwise}, this series is finite in the interior of
$\mathsf K_p$, so $\zeta_\ell\ell\in\partial\mathsf K_p$.
Each ray from $0$ meets this boundary exactly once, proving the remaining claims.
\end{proof}

\begin{proof}[Proof of Theorem~\ref{thm:main}]
Choose $t\in\partial\mathsf K_p$ with $t\cdot x=\xi_p(x)$ and set
$\ell=t/\|t\|$. By Lemma~\ref{lem:point-plane-duality},
$\zeta_\ell\ell=t$. On $\{0\leftrightarrow x\}$,
$H_\ell(C_0)\ge\ell\cdot x$, with $0<\ell\cdot x\le\norm{x}$.
Thus \eqref{eq:span-volume}, with $n=\norm{x}$ and $u=\ell\cdot x$,
gives \eqref{eq:main}. 
\end{proof}

\section{Extension of the Ornstein-Zernike theory}\label{section:OZ}
In this section, we are going to prove Theorem~\ref{theorem:OZ}. We are going to use and adapt the modern Ornstein-Zernike theory as developed in~\cite{CI02,AOV24, CIV08}.
\subsection{Coarse-graining}
\label{sec:skeleton}
\subsubsection{Skeleton}
First we construct the skeleton, same as in \cite[Section~2.2]{CIV08}, but boxes have different shape in order to minimize their intersection with the cluster of $0$. This is a crucial detail that will allow us to extract a cost of disconnecting these boxes from the cluster of $0$, which is exactly the point of Lemma~\ref{lem:deco}. Define
\[
 \mathbb E^d=\{\{u,v\}:u,v\in\mathbb Z^d,\ \|u-v\|_1=1\},
\]
Define
\[
\mathsf B_R(z)=(z+R\mathsf U_p)\cap\mathbb Z^d.
\]
Let $n=\norm{x}$, and $E_x(M):=\{x\in C_0,\ |C_0|\le Mn\}$, with $M$ big enough so that Theorem \ref{thm:main} applies.
Then
\begin{equation}
 0\le G_p(x)-\Pp(E_x(M))
 \le Ce^{-\xi_p(x)-cn^\beta}
\end{equation}
Finally, fix $ c_0:=\min_{\|y\|=1}\xi_p(y)>0,$ and
\[
 a=\max\{1,\xi_p(e_1),\ldots,\xi_p(e_d)\},
\]
For a vertex set $D\subset\mathbb Z^d$, define
\[
 \partial^{\mathrm{ext}}D
 =\{y\in\mathbb Z^d\setminus D:\
       \{y,z\}\in\mathbb E^d\text{ for some }z\in D\}.
\]
\paragraph{Algorithm of the coarse graining}
Fix $\eta>0$ small and an integer $K$ big.
Let $C$ be a finite cluster containing $0,x$.
For $z\in C$, set
\[
 b_C(z,r)=\#\{e\in\mathbb E^d:e\cap C\ne\varnothing,\
                                      |e\cap \mathsf B_r(z)|=1\}.
\]
This counts all crossing edges incident to $C$, including closed edges.
Fix the lexicographic order on vertices.
\begin{enumerate}
\item Set $x_0=0$ and $\mathcal C_{-1}=\varnothing$.
\item Given $x_i$, set
\[
 r_i=\min \operatorname*{arg\,min}_{r\in\{K,\ldots,K+\lfloor\eta K\rfloor\}}
                         b_C(x_i,r)\]
 Then set $A_i=\mathsf B_{r_i}(x_i)\setminus\mathcal C_{i-1}$, and $\mathcal C_i=\mathcal C_{i-1}\cup \mathsf B_{r_i}(x_i)$.

\item If possible, choose
\[
 x_{i+1}=\min\{y\in C\cap\partial^{\mathrm{ext}}\mathcal C_i:
 y\stackrel{C\setminus\mathcal C_i}{\longleftrightarrow}
                         \partial^{\mathrm{ext}}\mathsf B_K(y)\}.
\]
Return to step~2. Otherwise stop.
\end{enumerate}
For $i\ge1$, join $x_i$ to $x_{j(i)}$, where
\[
 j(i)=\min\{j<i:x_i\in\partial^{\mathrm{ext}}\mathsf B_{r_j}(x_j)\}.
\]
These big edges form the skeleton tree $\mathfrak T$, for which $m=N(\mathfrak T)+1$ is its number of vertices.
We also keep track of the chosen radius $r_i$ at each vertex in $\mathfrak T$.

\paragraph{The energy estimate.}
For $A\subset D$ finite, define
$b_D(A):=\#\{\{u,v\}\in\mathbb E^d:u\in A,\ v\notin D\}$.
Let $\mathbb P_p^D$ be independent Bernoulli bond percolation with parameter $p$ on the induced edge set $\mathbb E(D)=\{\{u,v\}\in\mathbb E^d:u,v\in D\}$, and let $\mathbb E_p^D$ be its expectation. For $z\in D$, $C_z^D$ is the open cluster of $z$ in this induced graph.
\begin{lemma}\label{lem:deco}
For $z\in D$ and every nonnegative function $f$ of $C_z^D$,
\begin{equation}\label{eq:extraction_locale}
 \mathbb E_p^D[(1-p)^{b_D(C_z^D)}f(C_z^D)]
 =\mathbb E_p[f(C_z)\,\mathbf1_{\{C_z\subset D\}}].
\end{equation}
In particular,
\begin{equation}\label{eq:local-energy}
 \mathbb E_p^D[(1-p)^{b_D(C_z^D)}\,\mathbf1_{\{
             \max_{y\in C_z^D}\xi_p(y-z)\ge K-a\}}]
 \le cK^de^{-K}.
\end{equation}
\end{lemma}
\begin{proof}
Given all internal edges,  $C_z^D$ is in fact exactly the entire cluster if and only if its $b_D(C_z^D)$ external edges that have one endpoint in $D^c$ are closed, which openness is independent of the internal cluster. This proves \eqref{eq:extraction_locale}.
Applying this to the indicator function leads to 
\[
 \sum_{\xi_p(y-z)\ge K-a}G_p(y-z)
 \le\sum_{\xi_p(y)\ge K-a}e^{-\xi_p(y)}
 \le cK^de^{-K},
\]
\end{proof}
Note that with the notation of the coarse-graining algorithm,
\begin{equation}\label{eq:boundary-cost}
 \sum_i b_{A_i}(C_{x_i}^{A_i})\le c|C|/(\eta K).
\end{equation}
Indeed, an edge crosses $\partial \mathsf B_r(x_i)$ for a bounded number
of integer values of $r$.Therefore, since $r_i$ is chosen in order to minimize $b$, there exists $c>0$ such that
\[
 b_C(x_i,r_i)\le
 \frac{c}{\eta K}|C\cap \mathsf B_{2K}(x_i)|.
\]
So we get \(\sum_i b_{A_i}(C_{x_i}^{A_i})\le c\sum_i b_C(x_i,r_i)\le c|C|/(\eta K).\)
Fix now a skeleton $\mathfrak T$.
Define $ B_i=b_{A_i}(C_{x_i}^{A_i}),$ and the event
\[
 F_i=\left\{
 \max_{y\in C_{x_i}^{A_i}}\xi_p(y-x_i)\ge K-a
 \right\}.
\]
For $n$ sufficiently large, on $\{E_x(M),\mathfrak T(C_0)=\mathfrak T\}$, all events $F_i$ occur,
and by \eqref{eq:boundary-cost}
\[
 \prod_i(1-p)^{B_i}\ge
 \exp\!\left\{-\frac{cMn}{\eta K}\right\}.
\]
Dropping the skeleton constraint after this bound, \eqref{eq:local-energy}
and independence of the edges inside the fixed, disjoint regions $A_i$ give
\begin{align*}
 \mathbb P_p(E_x(M),\mathfrak T(C_0)=\mathfrak T)&\le e^{cMn/(\eta K)}
 \mathbb E_p\!\left[\prod_i(1-p)^{B_i}\,\mathbf1_{\bigcap_i F_i}\right]\\
 &=e^{cMn/(\eta K)}\prod_i\mathbb E_p^{A_i}\!\left[(1-p)^{B_i}\,\mathbf1_{F_i}\right]\\
 &\le e^{cMn/(\eta K)}(cK^de^{-K})^m\\
 &\le\exp\{-Km+cm\log K+cMn/(\eta K)\},
\end{align*}
\paragraph{The entropy estimate.}
\begin{lemma}
\label{lem:entropy}
There are at most $\exp(cm\log K)$ skeletons with $m$ vertices.
\end{lemma}

\begin{proof}
Use the tree count of \cite[Lemma~4.1]{V} and multiply by at most
$(1+\lfloor\eta K\rfloor)^m$ choices of radii, still absorbed in
$\exp(cm\log K)$.
\end{proof}
\paragraph{Trunk and branches.}
Let $x_F$ be the first vertex $x_i$ of the skeleton such that
$x\in\mathsf B_{3K}(x_i)$. Such a vertex exists by the stopping rule for $\eta<1$ and $K$ sufficiently large.
The path from $0$ to $x_F$ is what we call the trunk, and we denote it by
$\mathfrak t=(\mathfrak t_0=0,\ldots,\mathfrak t_{N(\mathfrak t)}=x_F)$.
After deleting its edges, let $\mathfrak b_i$ be the branch rooted at $\mathfrak t_i$ and define, as in \cite{V}
\begin{align*}
 \mathfrak B&:=(\mathfrak b_0,\ldots,\mathfrak b_{N(\mathfrak t)}),\\
 N(\mathfrak B)&:=\sum_iN(\mathfrak b_i)=N(\mathfrak T)-N(\mathfrak t).
\end{align*}
Since consecutive vertices of the trunk are at $\xi_p$-distance at most
$(1+\eta)K+a$, one has
\begin{equation}\label{eq:trunk-length}
 N(\mathfrak t)\ge\frac{\xi_p(x)-3K}{(1+\eta)K+a}.
\end{equation}
\subsubsection{Geometry of typical trees: bound on the size}
\begin{lemma}
\label{lem:tree-size}
For every $\nu>0$ and $\varepsilon>0$, choose $\eta>0$ sufficiently small and $K$
sufficiently large. For all sufficiently large $n=\|x\|$,
\begin{equation}\label{eq:tree-size}
 \mathbb P_p\!\left(E_x(M),N(\mathfrak T)\ge(1+\nu)\xi_p(x)/K\right)
 \le Ce^{-(1+\nu/2)\xi_p(x)},
\end{equation}
and 
\begin{equation}\label{eq:branches-small}
 \mathbb P_p\!\left(E_x(M),N(\mathfrak B)\ge
                         \varepsilon\xi_p(x)/K\right)
 \le Ce^{-\xi_p(x)-c_\varepsilon n}.
\end{equation}
\end{lemma}

\begin{proof}
The proof is a routine energy-entropy argument and essentially the same as in \cite[Section~2]{CI02}. Here are the ideas :
Apply the energy--entropy argument of \cite[Section~2]{CI02}
using the energy estimate derived earlier, and Lemma~\ref{lem:entropy}.
The additional factor $\exp\{cMn/(\eta K)\}$ is absorbed by taking
$K$ large, since $\xi_p(x)\ge c_0n$.
For \eqref{eq:branches-small}, use the argument of
\cite[Section~2.4, proof of Lemma~2.2]{CIV08}:
by \eqref{eq:trunk-length}, for $\eta$ small and then $K,n$ large,
\[
 N(\mathfrak t)\ge(1-\varepsilon/2)\xi_p(x)/K.
\]
Thus $N(\mathfrak B)\ge\varepsilon\xi_p(x)/K$ implies
$N(\mathfrak T)\ge(1+\varepsilon/2)\xi_p(x)/K$, and conclude by applying \eqref{eq:tree-size}.
\end{proof}
Call trunks such that $N(\mathfrak t)\le\nu_1\xi_p(x)/K$, $\nu_1>1$ \textit{admissible}.
In this language, equation~\eqref{eq:tree-size} shows that inadmissible trunks can be neglected.
\subsubsection{Geometry of typical trees: refined estimates}
\paragraph{The surcharge function.}
For $t\in\partial\mathsf K_p$, define
\begin{align*}
 \mathfrak s_t(y)&=\xi_p(y)-t\cdot y,\\
 \mathfrak s_t(\mathfrak t)&=\sum_{i=1}^{N(\mathfrak t)}\mathfrak s_t(\mathfrak t_i-\mathfrak t_{i-1}),
\end{align*}
For $0<\delta<1$, define
\[
 \mathcal Y_\delta^{\blacktriangleleft}:=\{y\in\mathbb R^d:\mathfrak s_t(y)\le\delta\xi_p(y)\},
\]
and also $\mathcal Y_\delta^{\blacktriangleright}=-\mathcal Y_\delta^{\blacktriangleleft},~ 
 \mathcal Y_\delta^{\blacktriangleleft}(z)=z+\mathcal Y_\delta^{\blacktriangleleft},~ 
 \mathcal Y_\delta^{\blacktriangleright}(z)=z+\mathcal Y_\delta^{\blacktriangleright}.$
\begin{proposition}
\label{prop:surcharge}
For every $\varepsilon>0$, choose $\eta$ small and then $K$ large.
For large enough $n$, uniformly in $t\in\partial\mathsf K_p$,
\begin{equation}\label{eq:surcharge}
 \mathbb P_p(E_x(M),\mathfrak s_t(\mathfrak t)+KN(\mathfrak B)\ge\varepsilon n)
 \le e^{-t\cdot x-\varepsilon n/2}.
\end{equation}
\end{proposition}
\begin{proof}
Let $n=\norm{x}$ and define
\[
 Q:=\mathfrak s_t(\mathfrak t)+KN(\mathfrak B).
\]
Since $x_F\in\mathsf B_{3K}(x)$ and every trunk edge has
$\xi_p$-length at most $(1+\eta)K+a$, one has
\[
\begin{aligned}
 K\bigl(N(\mathfrak T)+1\bigr)
 &\ge \sum_{i=1}^{N(\mathfrak t)}
          \xi_p(\mathfrak t_i-\mathfrak t_{i-1})
          +KN(\mathfrak B)-(\eta K+a)N(\mathfrak t)+K\\
 &\ge t\cdot x+Q-(\eta K+a)N(\mathfrak t)-2K.
\end{aligned}
\]
Fix a sufficiently large constant $L$. For trees with
$N(\mathfrak T)\le Ln/K$, the energy estimate and
Lemma~\ref{lem:entropy} therefore gives
\[
\begin{aligned}
 &\mathbb P_p\bigl(E_x(M),Q\ge\varepsilon n,
                         N(\mathfrak T)\le Ln/K\bigr)\\
 &\quad\le
 (1+Ln/K)\exp\left\{-t\cdot x-\varepsilon n
       +cL\left(\eta+\frac{\log K}{K}\right)n
       +\frac{cMn}{\eta K}+cK\right\}.
\end{aligned}
\]
Choose $\eta$ small and then $K$ large so that the two terms
proportional to $n$ in the error are at most $\varepsilon n/4$.
For sufficiently large $n$, this bound is at most
$\tfrac12 e^{-t\cdot x-\varepsilon n/2}$.

For the remaining trees, the same energy--entropy estimate gives
\[
 \mathbb P_p\bigl(E_x(M),N(\mathfrak T)>Ln/K\bigr)
 \le C\exp\left\{-\frac L2n+\frac{cMn}{\eta K}\right\}.
\]
Since $|t\cdot x|= O(n)$, uniformly in
$t\in\partial\mathsf K_p$, taking $L$ sufficiently large makes this
bound at most $\tfrac12e^{-t\cdot x-\varepsilon n/2}$.
This proves \eqref{eq:surcharge}. The last assertion follows from
$t\cdot x=\xi_p(x)-\mathfrak s_t(x)$.
\end{proof}

\paragraph{Typical trunks have many cone-points.}
A trunk vertex $\mathfrak t_i$ is a forward cone-point if all its successors belong to $\mathcal Y_\delta^{\blacktriangleleft}(\mathfrak t_i)$, and a backward cone-point
if all its predecessors belong to $\mathcal Y_\delta^{\blacktriangleright}(\mathfrak t_i)$.
A cone-point of the trunk satisfies both conditions.
Define recursively
\begin{align*}
 \ell_1^{\blacktriangleleft}
  &=\min\{i:\mathfrak t_i\text{ is not a forward cone-point}\},\\
 r_k^{\blacktriangleleft}
  &=\min\{j>\ell_k^{\blacktriangleleft}:
        \mathfrak t_j\notin \mathcal Y_\delta^{\blacktriangleleft}(\mathfrak t_{\ell_k^{\blacktriangleleft}})\},\\
 \ell_{k+1}^{\blacktriangleleft} &=\min\{i\ge r_k^{\blacktriangleleft}:
                   \mathfrak t_i\text{ is not a forward cone-point}\}.
\end{align*}
Stop when the set is empty. Define the backward intervals by reversing
the order and using $\mathcal Y_\delta^{\blacktriangleright}$.
We say that a trunk vertex $\mathfrak t_i$ is marked if
\[
i\in
\bigcup_k
[\ell_k^{\blacktriangleleft},r_k^{\blacktriangleleft})
\;\cup\;
\bigcup_k
(r_k^{\blacktriangleright},\ell_k^{\blacktriangleright}].
\]
Let $\#_\delta^{\mathrm{marked}}(\mathfrak t)$ be the number of marked vertices.
Note that every unmarked vertex is a cone-point of the trunk.

\begin{lemma}[Corresponding to Lemma 4.4 of \cite{V}]
\label{lem:marked}
For $0<\delta<1/3$,
\begin{equation}\label{eq:marked-deterministic}
 \#_\delta^{\mathrm{marked}}(\mathfrak t)
 \le\frac{2\mathfrak s_t(\mathfrak t)}{\delta K}.
\end{equation}
For every $\varepsilon>0$, there exist $a_0,c,C>0$ and suitable
$\eta,K$ such that, if $\mathfrak s_t(x)\le a_0 n$,
\[
 \mathbb P_p(E_x(M),\#_\delta^{\mathrm{marked}}(\mathfrak t)\ge\varepsilon N(\mathfrak t))\le Ce^{-\xi_p(x)-cn}.
\]
\end{lemma}

\begin{proof}
Use the marked-interval argument of \cite[Lemma~4.4]{V}.
For the constant in \eqref{eq:marked-deterministic}, on each forward
marked interval $[i,j)$, define
$L_{i,j}=\sum_{k=i+1}^j\xi_p(\mathfrak t_k-\mathfrak t_{k-1})$. Then
\[
 \sum_{k=i+1}^j \mathfrak s_t(\mathfrak t_k-\mathfrak t_{k-1})=L_{i,j}-t\cdot(\mathfrak t_j-\mathfrak t_i)\ge\delta L_{i,j}>\delta K(j-i),
\]
since $\mathfrak t_j-\mathfrak t_i\notin \mathcal Y_\delta^{\blacktriangleleft}$ and
$\xi_p(\mathfrak t_j-\mathfrak t_i)\le L_{i,j}$.
Sum over the disjoint forward intervals and then the backward intervals.
The probability bound follows from Proposition~\ref{prop:surcharge}
and $N(\mathfrak t)\ge c_0n/(4K)$, as in the cited proof.
\end{proof}

\paragraph{Typical trees have many cone-points.}
A vertex $\mathfrak t_i$ of the trunk is a cone-point of $\mathfrak T$ when
\[
 \mathfrak T\subset \mathcal Y_{2\delta}^{\blacktriangleright}(\mathfrak t_i)\cup \mathcal Y_{2\delta}^{\blacktriangleleft}(\mathfrak t_i).
\]
A trunk cone-point is blocked if it doesn't satisfy this. Let
$\#_\delta^{\mathrm{blocked}}(\mathfrak T)$ be the number of blocked points.
We write $\#_\delta^{\mathrm{cone}}(\mathfrak T)$ for the number of cone-points
of the tree among the vertices of the trunk.

\begin{lemma}
\label{lem:blocked}
For a constant $c_\delta$,
\begin{equation}\label{eq:blocked-deterministic}
 \#_\delta^{\mathrm{blocked}}(\mathfrak T)\le c_\delta N(\mathfrak B).
\end{equation}
For every $\varepsilon>0$ and suitable $\eta,K$,
\[
 \mathbb P_p(E_x(M),\#_\delta^{\mathrm{blocked}}(\mathfrak T)\ge\varepsilon N(\mathfrak t))\le Ce^{-\xi_p(x)-cn}.
\]
\end{lemma}

\begin{proof}

We use same argument as in the proof of
\cite[Section~2.7, Lemma~2.5]{CIV08}.
Let $X_i$ be the number of cone-points of the trunk blocked
by the branch $\mathfrak b_i$. By that argument,
$X_i\le c_\delta N(\mathfrak b_i)$.
Since every blocked cone-point is blocked by at least one branch,
\[\#_\delta^{\mathrm{blocked}}(\mathfrak T)\le \sum_i X_i \le c_\delta\sum_i N(\mathfrak b_i)=c_\delta N(\mathfrak B),
\]
which proves \eqref{eq:blocked-deterministic}.

For $\eta$ sufficiently small, $K$ sufficiently large, and
$n=\norm{x}$ sufficiently large, \eqref{eq:trunk-length} gives
$N(\mathfrak t)\ge \xi_p(x)/(4K)$. Consequently,
\[
 \left\{\#_\delta^{\mathrm{blocked}}(\mathfrak T)
 \ge\varepsilon N(\mathfrak t)
 \right\}
 \subset
 \left\{
 N(\mathfrak B)\ge
 \frac{\varepsilon}{4c_\delta}\frac{\xi_p(x)}{K}
 \right\}.
\]
Applying \eqref{eq:branches-small} with
$\varepsilon/(4c_\delta)$ proves the probability bound.
\end{proof}

\begin{theorem}
\label{thm:tree-cones}
For every $\varepsilon>0$ and $0<\delta<1/3$, small enough $\eta$ and big enough $K$ give
\[
 \mathbb P_p(E_x(M),\#_\delta^{\mathrm{cone}}(\mathfrak T)\le(1-\varepsilon)N(\mathfrak t))\le Ce^{-\xi_p(x)-cn}
\]
whenever $\mathfrak s_t(x)\le a_0n$, for a sufficiently small $a_0>0$.
Moreover, the number of trunk vertices which are either marked or
blocked is at most
\begin{equation}\label{eq:tree-cones}
 c_\delta\bigl(K^{-1}\mathfrak s_t(\mathfrak t)+N(\mathfrak B)\bigr).
\end{equation}
\end{theorem}

\begin{proof}
We combine the marked-point and blocked-point estimates of
\cite[Sections~2.6--2.8, Lemmas~2.4--2.5]{CIV08}.

Every unmarked and unblocked trunk vertex is a cone-point
of the tree. By the deterministic estimates in
Lemmas~\ref{lem:marked} and~\ref{lem:blocked},
\begin{align*}
 \#\{i:\mathfrak t_i\text{ is marked or blocked}\}
 &\le \#_\delta^{\mathrm{marked}}(\mathfrak t)
      +\#_\delta^{\mathrm{blocked}}(\mathfrak T)\\
 &\le c_\delta\bigl(
 K^{-1}\mathfrak s_t(\mathfrak t)+N(\mathfrak B)\bigr),
\end{align*}
which proves \eqref{eq:tree-cones}.

Moreover, since the trunk has $N(\mathfrak t)+1$ vertices,
\begin{align*}
 &\left\{
 \#_\delta^{\mathrm{cone}}(\mathfrak T)
 \le(1-\varepsilon)N(\mathfrak t)
 \right\}\\
 &\quad\subset
 \left\{
 \#_\delta^{\mathrm{marked}}(\mathfrak t)
 \ge\frac{\varepsilon}{2}N(\mathfrak t)
 \right\}
 \cup
 \left\{
 \#_\delta^{\mathrm{blocked}}(\mathfrak T)
 \ge\frac{\varepsilon}{2}N(\mathfrak t)
 \right\}.
\end{align*}
Choose the existing $\eta,K$ of Lemmas~\ref{lem:marked} and~\ref{lem:blocked} and take $a_0$ sufficiently small. A union bound then gives the claimed estimate whenever $\mathfrak s_t(x)\le a_0n$.
\end{proof}

\subsection{Geometry of typical clusters}
\subsubsection{Cone-points of \texorpdfstring{$C$}{C}}
\label{sec:surgery}

\paragraph{Cones.}
For fixed $t$, choose $e\in\{\pm e_1,\ldots,\pm e_d\}$ such that
$r=t\cdot e=\max_j|t\cdot e_j|$, and write  $y^\parallel:=t\cdot y/r$, and 
\begin{align*}
 \mathcal Y^+&=\mathcal Y_{3\delta}^{\blacktriangleleft},\\
 \mathcal Y^-&=-\mathcal Y^+.
\end{align*}

\paragraph{Cone-points of the cluster.}

For a finite connected graph $C$, a vertex $v$ is a \emph{cone-point of $C$}
if both of the following conditions hold:
\[
 \begin{gathered}
 C\subset(v+\mathcal Y^-)\cup(v+\mathcal Y^+),\\
 C\cap\{z:v^\parallel-1\le z^\parallel\le v^\parallel+1\}
 =\{v-e,v,v+e\},
\end{gathered}
\]
which is nothing else than the definition of Campanino and Ioffe~\cite[Section~4.1]{CI02}. We say that $v$ is a \textit{forward cone-point} of $C$ if 
$$ C\subset(v+\mathcal Y^+),$$
and 
$$C\cap\{z:v^\parallel-1\le z^\parallel\le v^\parallel+1\}
 =\{v,v+e\}.$$
We say that $v$ is a \textit{backward cone-point} of $C$ if 
$$ C\subset(v+\mathcal Y^-),$$
and 
$$C\cap\{z:v^\parallel-1\le z^\parallel\le v^\parallel+1\}
 =\{v,v-e\}.$$
Let $\#_\delta^{\mathrm{cone}}(C;0,x)$ count the cone-points.
\begin{theorem}
\label{prop:cones}\label{lemma:few_conepoints}
There exist $a_0,\rho,c,C>0$ such that, if $\mathfrak s_t(x)\le a_0\|x\|$,
\begin{align}
 \mathbb P_p(E_x(M),\#_\delta^{\mathrm{cone}}(C_0;0,x)\le\rho\|x\|)
 &\le Ce^{-\xi_p(x)-c\|x\|},\label{eq:cones-thin}\\
 \mathbb P_p(x\in C_0,\ |C_0|<\infty,\
                  \#_\delta^{\mathrm{cone}}(C_0;0,x)\le\rho\|x\|)
 &\le Ce^{-\xi_p(x)-c\|x\|^\beta}.
\end{align}
The constants are uniform in $t\in\partial\mathsf K_p$.
\end{theorem}
\begin{proof}
The exact same proof as in~
\cite[Section~2.9, Theorem~2.2]{CIV08} shows \eqref{eq:cones-thin}
Then,
\begin{align*}
 &\mathbb P_p\bigl(x\in C_0,\ |C_0|<\infty,\
       \#_\delta^{\mathrm{cone}}(C_0;0,x)\le\rho n\bigr)\\
 &\quad\le
 \mathbb P_p\bigl(E_x(M),\
       \#_\delta^{\mathrm{cone}}(C_0;0,x)\le\rho n\bigr)
 +\mathbb P_p(x\in C_0,\ Mn<|C_0|<\infty)\\
 &\quad\le Ce^{-\xi_p(x)-cn^\beta},
\end{align*}
which is what we wanted.
\end{proof}
\emph{Remark : } Notice that we actually recover a form of separation of masses on irreducible connections.
\subsubsection{Decomposition of \texorpdfstring{$C$}{C} into irreducible pieces}
\label{sec:pieces}
\paragraph{Irreducible pieces.}
We are now going to associate a directed random walk to every cluster and study its properties. 
Fix $t_0\in\partial\mathsf K_p$.
The cones $\mathcal Y^\pm$ and $e$ are
those defined from $t_0$.
Define the height of a vertex $z$ to be 
\[
 z^\parallel :=\frac{t_0\cdot z}{t_0\cdot e}.
\]
We denote by $\mathcal I_{\mathrm b}$, $\mathcal I_{\mathrm f}$ and $\mathcal I$
the sets of backward-irreducible, forward-irreducible and irreducible pieces,
respectively, defined as follows.
\begin{itemize}
\item $C\in \mathcal I_{\mathrm b}$ if $0\in C$, $v$ is a backward cone-point, $v^\parallel\ge0$, and $v$ is the only cone-point of nonnegative height. Define $X(C):=v$.
\item $C\in \mathcal I_{\mathrm f}$ if $0\in C$, $v$ is a forward cone-point of $C$, $v^\parallel\le0$, and $v$ is the only cone-point of nonpositive height. Define $X(C):=-v$.
\item $C\in \mathcal I$ if $0$ is its only forward cone-points and $v$ its only backward cone-point, $v^\parallel>0$. Define $X(C)=v$.
\end{itemize}

\begin{lemma}
A finite cluster containing $0,x$, with at least two cone-points
between their heights, has a unique decomposition
\[
 C=C_{\mathrm b}\sqcup C_1\sqcup\cdots\sqcup C_N\sqcup C_{\mathrm f},
\]
where $N\ge 1$ and $C_{\mathrm b}\in \mathcal I_{\mathrm b}$, $C_{\mathrm f}\in \mathcal I_{\mathrm f}$, $C_i\in \mathcal I$ are translated and joined
at their marked endpoints, and
$x=X(C_{\mathrm b})+\sum_iX(C_i)+X(C_{\mathrm f})$.
Conversely every such concatenation gives one of these clusters.
\end{lemma}

\begin{proof}
Same as \cite[Section~2.10]{CIV08}.
\end{proof}

\paragraph{The weights.}
We define the following weights
\begin{equation}
 \widetilde{\mathsf w}(C)=
 \begin{cases}
 (1-p)^{-2d}\mathbb P_p(C_0=C),&C\in \mathcal I,\\
 (1-p)^{-d}\mathbb P_p(C_0=C),&C\in \mathcal I_{\mathrm b}\sqcup \mathcal I_{\mathrm f}.
 \end{cases}
\end{equation}
This gives us  : 
\begin{equation}\label{eq:factorisation}
 \mathbb P_p(C_0=C)=\widetilde{\mathsf w}(C_{\mathrm b})\widetilde{\mathsf w}(C_{\mathrm f})\prod_{i=1}^N\widetilde{\mathsf w}(C_i).
\end{equation}

We then consider the weights $\mathsf w_t(C)=e^{t\cdot X(C)}\widetilde{\mathsf w}(C)$ and
\begin{align*}
 \mathsf p_t(y)&:=\sum_{\substack{C\in \mathcal I\\X(C)=y}}\mathsf w_t(C),\\
 \rho_{\mathrm b,t}(y)&:=\sum_{\substack{C\in \mathcal I_{\mathrm b}\\X(C)=y}}\mathsf w_t(C),\\
 \rho_{\mathrm f,t}(y)&:=\sum_{\substack{C\in \mathcal I_{\mathrm f}\\X(C)=y}}\mathsf w_t(C).
\end{align*}

The following proposition provides the directed random walk representation of \(G_p(x)\) : 

\begin{proposition}\label{lemma:random_walk}
For $t\in\mathsf K_p$ sufficiently close to $t_0$, there exists \(C,c>0\) such that 
\[\vert r_t(x)\vert \leq Ce^{-c\norm{x}^\beta} \]
with 
\[
e^{t\cdot x}G_p(x) - \sum_{n\geq 1}\sum_{a,b,y_1,...,y_n}\mathbf1_{\{a+y_1+...+y_n+b=x\}}\rho_{\mathrm b,t}(a)\rho_{\mathrm f,t}(b)\prod_{i=1}^n \mathsf p_t(y_i)=:r_t(x)
\]
\end{proposition}
\begin{proof}
The result follows directly from~\eqref{eq:factorisation} and the control on \(r_t(x)\) follows from Theorem~\ref{lemma:few_conepoints}
\end{proof}
We will need several properties of the weights above. Firstly, we will show that they inherit from the stretched exponential decay. Moreover, $\mathsf p_t$ actually defines a probability measure on \(\bbZ^d\) for $t\in\partial\mathsf K_p$ sufficiently close to $t_0$. This will in particular imply that all of the moments of the displacements exist, and hence we will be able to use standard probability results to conclude Theorem~\ref{theorem:OZ}.
\begin{proposition}\label{proposition:tilded}

For $t\in\mathsf K_p$ sufficiently close to $t_0$, uniformly in $t,y$,
\begin{equation}\label{eq:tilted-tails}
 \mathsf p_t(y)+\rho_{\mathrm b,t}(y)+\rho_{\mathrm f,t}(y)
 \le Ce^{-c\|y\|^\beta}.
\end{equation}
For $t\in\partial\mathsf K_p$ in this neighbourhood,
\begin{equation}
 H(t):=\sum_y\mathsf p_t(y)=1.
\end{equation}
\end{proposition}

\begin{proof}[Proof of~\eqref{eq:tilted-tails}]
We first consider $\mathsf p_t$. We consider two cases : if $\surcharge_{t_0}(y)>a_0\norm{y}$, then, for
$\norm{t-t_0}<a_0/2$,
\[
\mathsf p_t(y)\leq (1-p)^{-2d}e^{t\cdot y}G_p(y)\leq (1-p)^{-2d}e^{-a_0\norm{y}/2}.
\]
Otherwise, since irreducible pieces have exactly two
cone-points, Theorem~\ref{lemma:few_conepoints} gives
for $\norm{y}$ sufficiently large,
\[
 \mathsf p_t(y)\leq Ce^{t\cdot y-\xi_p(y)-c\norm{y}^{\beta}}
 \leq Ce^{-c\norm{y}^{\beta}},
\]
where we used $t\in\mathsf K_p$.

The same bounds hold for $\rho_{\mathrm b,t}$ and $\rho_{\mathrm f,t}$, and the proof is a classical OZ theory argument.
Let us define the following generating functions : 
\begin{align*}
 G(t)&:=\sum_{x\in\bbZ^d}e^{t\cdot x}G_p(x),\\
 R(t)&:=\sum_{x\in\bbZ^d}r_t(x),
\end{align*}
and \( H(t):=\sum_{y\in\bbZ^d}\mathsf p_t(y), ~B(t):=\sum_{y\in\bbZ^d}\rho_{\mathrm b,t}(y), ~F(t):=\sum_{y\in\bbZ^d}\rho_{\mathrm f,t}(y).\)
For any \(t\) in the interior of \(\mathsf K_p\), all of these series converge. Hence, using the decomposition above, we get the following renewal equation : 
\begin{align*}
 G(t)&=R(t)+B(t)F(t)\sum_{n\geq1}H(t)^n\\
 &=R(t)+B(t)F(t)\dfrac{H(t)}{1-H(t)}.
\end{align*}
Notice now that by estimates above, for \(t\in\partial \mathsf K_p\) close to \(t_0\), \(R(t),B(t),F(t)\) are all finite and \(G(t)=\infty\) by Lemma~\ref{lem:point-plane-duality}. Together with $H(t)\le1$, this implies that \(H(t)=1\) finishing the proof.
\end{proof}
For $t\in\partial\mathsf K_p$ sufficiently close to $t_0$, write
$\mathbb P_{t,0}$ for the law of the random walk started at $S_0=0$
with increment law $\mathsf p_t$, and $\mathbb E_{t,0}$ for its expectation.

We will now study a random walk \(S_n=\sum_{i=1}^{n}X_i\) where \(X_i\) are sampled according to \(\mathsf p_t\). We will first show that the expected value of these random variables point to the direction \(x/\norm{x}\). In previous works, this was done using the exponential moments coming from the separation of masses, which is unavailable here, but we give a simpler argument.

\begin{proposition}\label{prop:drift}
For $x\ne0$ and $t\in\partial\mathsf K_p$ sufficiently close to $t_0$ with $t\cdot x=\xi_p(x)$, there exists \(c>0\) such that
\[\mathbb E_{t,0}[X_1]=c\cdot\dfrac{x}{\norm{x}}\]
\end{proposition}
\begin{proof}
Firstly, one can use Taylor's theorem on \(H\) for vectors on the boundary of \(\partial \mathsf K_p\) since the sum \(\sum_y \norm{y}^k \mathsf p_t(y)\) is uniformly bounded. Set $s=x/\|x\|$ and $m=\mathbb E_{t,0}[X_1]$.
Since every irreducible displacement has positive $t_0$ projection, one has $t_0\cdot m>0$. Consider now 
\[
 H(u)=\sum_y \mathsf p_t(y)e^{(u-t)\cdot y}.
\]
The preceding renewal argument implies that $H(u)\leq1$ in $\mathsf K_p$ near $t$, with equality on its boundary.
By Taylor near $t$,
\[
 H(u)=1+m\cdot(u-t)+O(\|u-t\|^2).
\]
In particular, if \(u\) is on the boundary one has \[m\cdot(u-t)=O(\|u-t\|^2).\]
 
Suppose $s\neq m/\|m\|$ and define $v=s-m/\|m\|$.
Direct calculation shows that $s\cdot v >0$, while $m\cdot v<0.$
For small enough $\varepsilon>0$, $(1-\varepsilon)t$ lies inside $\mathsf K_p$, and $t+\varepsilon v$ lies outside of it.
The segment joining their endpoints consequently contains a boundary point 
$$u_\varepsilon:=(1-\theta_\varepsilon)(1-\varepsilon)t
   +\theta_\varepsilon(t+\varepsilon v)$$
for some $\theta_\varepsilon\in (0,1)$.
Then, $u_\varepsilon-t$ must have a projection on $m$ that goes in the opposite direction of $m$, and that has a norm of order $\varepsilon$.
Indeed,
\begin{align*}
 m\cdot(u_\varepsilon-t)
 &=-\varepsilon\bigl((1-\theta_\varepsilon)m\cdot t
                     +\theta_\varepsilon(-m\cdot v)\bigr)\\
 &\le-\varepsilon\bigl((1-\theta_\varepsilon)c_1
                      +\theta_\varepsilon c_1\bigr)\\
 &=-c_1\varepsilon,
\end{align*}
where $c_1 := \min \{m\cdot t, -m\cdot v \}$, which is strictly positive for $t$ sufficiently close to $t_0$.
Since $u_\varepsilon\in\partial\mathsf K_p$, the preceding Taylor bound gives
$m\cdot(u_\varepsilon-t)=O(\varepsilon^2)$, a contradiction for sufficiently
small $\varepsilon>0$. 
\end{proof}

\begin{proposition}\label{lemma:TLL}
Consider a random walk \(S_n=\sum_{i=1}^n X_i\), with law $\mathbb P_0$, expectation $\mathbb E_0$ and $S_0=0$, where \(X_i\in\bbZ^d\) are i.i.d and \(\mathbb E_0(\norm{X_1}^l)<\infty\) for every integer \(l\ge1\). Assume that \(\mathbb E_0(X_1)=cs\) with \(c>0\) and \(s\in\mathbb{S}^{d-1}\) and that the covariance matrix of \(X_1\) is non-degenerate. Assume also that the walk is aperiodic. Then, as $k\to\infty$ with $ks\in\mathbb Z^d$, there exists \(C>0\) :
\[\sum_{N\geq 1}\mathbb P_0(S_N=ks)=\dfrac{C}{k^{(d-1)/2}}(1+o_k(1)).\]
\end{proposition}

\begin{proof}
Fix $0<\varepsilon<1/2$. We consider two regimes :
$|N-k/c|\le k^{1/2+\varepsilon}$ and $|N-k/c|>k^{1/2+\varepsilon}$.
Write $\Sigma$ for the covariance matrix of $X_1$ and define
\[
 g_N(z)=\frac{1}{(2\pi N)^{d/2}\sqrt{\det\Sigma}}
 \exp\!\left\{-\frac{(z-Ncs)^{\mathsf T}\Sigma^{-1}(z-Ncs)}{2N}\right\}.
\]
Under our moment and aperiodicity assumptions, the quantitative local
central limit theorem gives
\[
 \sup_{z\in\mathbb Z^d}|\mathbb P_0(S_N=z)-g_N(z)|
 \le C_1N^{-(d+1)/2}.
\]
In the first regime, summing this error gives
$O(k^{-d/2+\varepsilon})=o(k^{-(d-1)/2})$. A Gaussian Riemann sum,
with $v=(N-k/c)/\sqrt{k}$, therefore yields
\begin{align*}
 \sum_{|N-k/c|\le k^{1/2+\varepsilon}}\mathbb P_0(S_N=ks)
 &=\frac{C}{k^{(d-1)/2}}(1+o(1)),\\
 C&=\frac{c^{(d-3)/2}}
 {(2\pi)^{(d-1)/2}\sqrt{\det\Sigma}\sqrt{s^{\mathsf T}\Sigma^{-1}s}}>0.
\end{align*}

For the second regime, for every even integer $q\ge4$,
\begin{align*}
 \mathbb E_0\norm{S_N-csN}^q&\le C_qN^{q/2},\\
 \mathbb P_0(S_N=ks)&\le\frac{C_qN^{q/2}}{|N-k/c|^q}.
\end{align*}
Summing this bound separately over $k/(2c)\le N\le2k/c$ and its
complement gives
\[
 \sum_{|N-k/c|>k^{1/2+\varepsilon}}\mathbb P_0(S_N=ks)
 \le C_q\bigl(k^{1/2+\varepsilon-q\varepsilon}+k^{1-q/2}\bigr).
\]
Taking $q$ sufficiently large makes this $o(k^{-(d-1)/2})$, which
finishes the proof.
\end{proof}
Now, we are in position to prove Theorem~\ref{theorem:OZ}
\begin{proof}[Proof of Theorem~\ref{theorem:OZ}]
Fix \(x\in\bbZ^{d}\) and \(t\in\partial\mathsf K_p\) dual to \(x\). Firstly, thanks to Proposition~\ref{lemma:random_walk}, we have:
\begin{equation*}
e^{\xi_p(x)}G_p(x)=\sum_{a,b}\sum_{n\geq 1}\rho_{\mathrm b,t}(a)\rho_{\mathrm f,t}(b)\mathbb P_{t,0}(S_n=x-a-b) + O(e^{-c\norm{x}^\beta}).
\end{equation*}
Moreover, using Proposition~\ref{proposition:tilded}, we assume (up to a negligible error) that \(\norm{a}+\norm{b}\leq \norm{x}^{1/4}\). Therefore, this means that \(x-a-b=x+o(\sqrt{\norm{x}})\), and we get that 
\[\sum_{n\geq1}\mathbb P_{t,0}(S_n=x-a-b)=\dfrac{C}{\norm{x}^{\frac{d-1}{2}}}(1+o_{\norm{x}}(1))\]
Overall, since the sums above over \(a,b\) converge, we can conclude that :
\[
 G_p(x)=
 \frac{\Psi_p(x/\|x\|)}{\|x\|^{(d-1)/2}}e^{-\xi_p(x)}(1+o(1)),
\]
\end{proof}
 
\section{Acknowledgments}
K.K. gratefully acknowledges the financial support of the Swiss National
Science Foundation (SNSF) under grant number $200021\_219333$. K.K. is a member of the NCCR SwissMAP.
Both authors are deeply grateful to Yvan Velenik for his detailed reading of the manuscript, his many suggestions for improving its presentation, and his corrections of several errors. We also thank him for being an outstanding PhD supervisor. His patience and skill as a teacher were essential to our understanding of Ornstein–Zernike theory. Without his patience and all that he has taught us, neither this paper nor any of our other papers would have been possible.

\paragraph{Use of AI tools.} We used ChatGPT to check proofs, identify relevant references, and improve the writing. All the mathematical ideas, proof strategies and proofs were developed by the authors. The authors take full responsibility for the final manuscript.

\end{document}